\documentclass{article}
\usepackage{amssymb}
\usepackage{amsmath}
\usepackage{amsfonts}
\usepackage{afterpage}
\usepackage[pdftex]{graphicx}
\usepackage{amstext}
\usepackage{cite}
\usepackage{color}

\usepackage[a4paper,
top=1.5in, left=1in, right=1in, includefoot]{geometry}

\newtheorem{theorem}{Theorem}

\newtheorem{corollary}[theorem]{Corollary}

\newtheorem{definition}[theorem]{Definition}

\newtheorem{lemma}[theorem]{Lemma}
\newtheorem{notation}[theorem]{Notation}

\newtheorem{proposition}[theorem]{Proposition}

\newtheorem{definitions}[theorem]{Definitions}
\newtheorem{Remark}[theorem]{Remark}
\newenvironment{proof}[1][Proof]
{\noindent\textbf{#1.} }{\ \rule{0.5em}{0.5em}}

\begin{document}

\title{Irrational eigenvalues of D-Dimensional Cellular automata}
\author{ $^{1}$ Nassima Ait Sadi and $^{2}$ Rezki Chemlal  \\
$^{1}$$^{,}$$^{2}$ Laboratory of applied mathematics, University of Bejaia,Algeria.\\
$^{2}$ National Higher School in Mathematics, Algiers.
Algeria.\\
\textbf{Emails:} $^{1}$nassima.aitsadi@univ-bejaia.dz, $^{2}$%
rezki.chemlal@nhsm.edu.dz}

\date{}
\maketitle

\begin{abstract}

Cellular automata are dynamical systems defined on lattices and commuting with the Bernoulli shift. In this work, we focus on the spectral properties of D-dimensional cellular automata. We give a characterization of their spectrum from both topological and ergodic point of view. The main results of the paper show the impossibility for a cellular automaton with a fully blocking pattern to have a measurable irrational eigenvalues. Further more, a cellular automaton with a set of equicontinuity points of positive measure cannot have a measurable irrational eigenvalue. 

\begin{description}
\item[Keywords:] Cellular automata, eigenvalues, ergodic theory.
\item[AMS Subject Classification:] 37B15
 
\end{description}
\end{abstract}

\section{Introduction}

Cellular automata, originating from Von Neumann in the late 1940s and initially inspired by biological models, have evolved into versatile tools applied in various modeling contexts. They serve as models for complex systems in a wide range of areas .

A Cellular automaton is a discrete model of computation that evolve in space
and time. It consists of a regular grid of cells that can only adopt only
one given state at each time unit. For each cell, a set of cells called its
neighborhood is defined relative to the specified cell and it state at each
time unit is determined by applying the associated local rule of the
cellular automaton to the current state of the cell and the states of the
cells in which is connected. Typically, the rule for updating the state of
cells is the same for each cell and does not change over time. A
configuration is a snapshot of the state of all automata in the lattice. The
lattice is usually $\mathbb{Z}^{d}$ where $d$ represents the dimension of the cellular automaton.

One way to study the dynamical properties of a cellular automaton is to
endow the space of configurations with a topology and consider the cellular
automaton as a discrete dynamical system and by equipping the space of
configurations with the uniform Bernoulli measure \cite{Pivato, Walters}, cellular automata can then be studied
from the ergodic theory point of view. 
The uniform Bernoulli measure plays a central
role in the ergodic theory of cellular automata. This measure is invariant if and only if the cellular automaton is surjective \cite{Hedlund}.

In this paper, we study the eigenvalues of D-dimensional cellular automata which formalize some kind of structure representing the nonchaotic part of the dynamics. They correspond to complex rotations that are factors of the system. 
We start by showing that any $D-$dimensional cellular automaton with equicontinuity points admits at least one periodic factor. After that we
characterize the spectrum of $D-$dimensional cellular automata from the topological point of view first and then we show that if the
measure of the set of points of surjective cellular automaton with an eventually periodic trace (we mean by a trace of a given point $x$ studying
its evolution in the central window of radius $(2r)^{d}$ under the iteration of the cellular automaton) is equal to one then the cellular automaton cannot have measurable irrational eigenvalues. 
Furthermore, We show that if the measure of the set of equicontinuity points of $F$ is of positive measure, then either the cellular automaton $F$  have a measurable rational eigenvalue or the associated eigenfunction is null on the set of equicontinuity points. We finish by giving two applications of this result, one in the case when the cellular automaton is equicontinuous, an other one when the cellular automaton admits at least one fully blocking pattern.

\section{Basic definitions}
In this section, we present some basic notions in symbolic dynamics used along this paper.

\begin{notation}
    \begin{itemize}
        \item  A is a finite alphabet.
        \item The interval $[a,b]$ represent the set of integers from $a$ to $b$.
       
        \item The size of a pattern $w\in A^{i_{1}\times i_{2}\times ...\times i_{d}}$ is 
        $|w|=\prod_{j=1}^{d} |w|_{j}=\prod_{j=1}^{d} i_{j}.$ 
        \item  Denote by $A^{\mathbb{Z}^{d}}$ the set of all configurations in $\mathbb{Z}^{d}$ constructed over the alphabet $A$.
        \item For any pattern $w$, we define the cylinder $[w]$ at position $\overrightarrow{i}=(i_{1},i_{2},...,i_{d})\in \mathbb{Z}^{d}$ by:
\begin{equation*}
\lbrack w]_{\overrightarrow{i}}=\{x\in A^{\mathbb{Z}^{d}}:x|_{\prod\limits_{s=1}^{d}[i_{s},i_{s}+|w|_{s}]}=w\}
\end{equation*}
    \item The cylinder at the origin is denoted by $[w]_{\overrightarrow{0}}$.
    \end{itemize}
\end{notation}

\begin{definition} We define a metric over $A^{\mathbb{Z}^{d}}$ by setting $d(x,y)=0$ if $x=y$ and for all $x\neq y\in A^{\mathbb{Z}^{d}}$ :
\begin{equation*}
d(x,y)=2^{-\min \{||\overrightarrow{i}||:x_{\overrightarrow{i
}}\neq y_{\overrightarrow{i}}\}}\text{where }||\overrightarrow{i}||=\underset{s=\overline{1,D}}{\max }\{|i_{s}|\}.
\end{equation*}
The space $A^{\mathbb{Z}^{d}}$ endowed with $d$ is compact, perfect, and totally disconnected in the product topology \cite{Kurka}. 
\end{definition}
\begin{itemize}
\item Any vector $\overrightarrow{i}\in \mathbb{Z}^{d}$ determines a continuous shift map
$\sigma _{\overrightarrow{i}}:A^{\mathbb{Z}^{d}}\rightarrow A^{\mathbb{Z}^{d}}$ defined by\newline $\sigma _{\overrightarrow{i}}(x)_{\overrightarrow{j}}=x_{
\overrightarrow{j+i}},\forall \overrightarrow{j}\in \mathbb{Z}^{d}.$
\end{itemize}

\subsection{Cellular Automata}

\begin{definition}
    Let $\Bbbk $ be a subset in $\mathbb{Z}^{d}$ and $f:A^{\Bbbk }\rightarrow A$ be a function called \emph{the local rule}. \newline
A cellular automaton determined by $f$ is the function $F:A^{\mathbb{Z}^{d}}\rightarrow A^{\mathbb{Z}^{d}}$ defined by :
\begin{center}
    
$F(x)_{\overrightarrow{m}}=f(x_{\overrightarrow{m}+\Bbbk
}) $ for all $x\in A^{\mathbb{Z}^{d}}$ and all $\overrightarrow{m}\in \mathbb{Z}^{d}.$
\end{center}

\end{definition}

 Curtis-Hedlund and Lyndon \cite{Hedlund} showed that cellular automata are exactly the continuous transformations of $A^{\mathbb{Z}^{d}}$ that commute with all shifts. We refer to $\Bbbk $ as a radius of $F.$
\begin{definitions}
    \begin{itemize}
        \item A point $x$ is a periodic point of period $p$ if $F^{p}(x)=x$ and $\forall
i<p,$ $F^{i}(x)\neq x.$ 
\item If there exist an integer $m$ such that $F^{m}(x)$
is periodic then $x$ is called an eventually periodic point.
\item The periodic points of the shift are called spatially periodic points and the periodic points for $F$ that are not shift periodic are called strictly temporally periodic points. Any spatially periodic point is eventually periodic for the cellular automaton.

\item A pattern $w$ of size $(k_{1}\times k_{2}\times...k_{d}$) is called $(r_{1},r_{2},...,r_{d})-$blocking with offset $(p_{1},p_{2},...,p_{d})$ if
there exist a non-negative integers $\{p_{i}\}_{i \in [1,d]}$ satisfying 
$p_{i}\leq k_{i}-r_{i}$ for all $i \in [1,d]$ such that for any $x,y\in
\lbrack w]_{\overrightarrow{0}}$ and $n\geq 0$, we have: $%
F^{n}(x)|_{\prod\limits_{j=1}^{d}[p_{i},p_{i}+r_{i})}=F^{n}(y)|_{\prod%
\limits_{j=1}^{d}[p_{i},p_{i}+r_{i})}.$\newline
When the offset $p_{i}=0,\forall i \in [1,d]$ the pattern $w$ is said a fully blocking pattern.
\begin{figure}[!hbtp]
\centering
\begin{minipage}{0.49\textwidth}
    \centering
    \includegraphics[width=6.5cm,height=4cm]{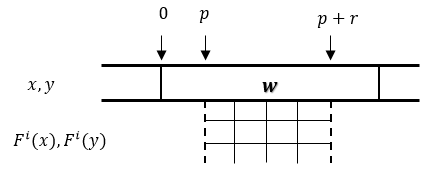}
     \caption{Blocking word, for d=1}
     \label{}  
     \end{minipage}
\hfill  
\begin{minipage}{0.49\textwidth}
    \centering
    \includegraphics[width=6cm,height=4cm]{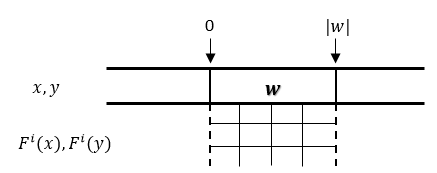}
    \caption{Fully blocking word, for d=1}
    \label{u}
    \end{minipage}
\hfill
\end{figure}

\item A point $x$ is a point of equicontinuity for the cellular automaton $(A^{\mathbb{Z}^{d}},F)$ if 
\begin{equation*}
\forall \varepsilon >0,\exists \delta >0:\forall y:d(x,y)<\delta \Rightarrow
\forall n\in \mathbb{N},d(F^{n}(x),F^{n}(y))<\varepsilon
\end{equation*}

\item The cellular automaton $F$ is equicontinuous if all points are points of equicontinuity and is almost equicontinuous if the set of equicontinuity contain a countable intersection of dense open sets.

\item  The cellular automaton $F$ is said to be sensitive to the initial conditions if
\begin{equation*}
\exists \varepsilon >0,\forall x\in A^{\mathbb{Z}^{d}},\forall \delta >0,\exists y:d(x,y)<\delta ,\exists n\in \mathbb{N}:d(F^{n}(x),F^{n}(y))\geq \varepsilon.
\end{equation*}

    \end{itemize}
\end{definitions}

Kurka \cite{Kurka} introduced a classification of one dimensional cellular automata according to sensitivity and equicontinuity of the cellular automaton.

\begin{theorem}
Let $F$ be a one dimensional cellular automaton with radius $r.$ The following properties are equivalent:

\begin{enumerate}
\item $F$ is not sensitive.
\item $F$ admit an $r-blocking$ pattern.
\item $F$ is almost equicontinuous.
\end{enumerate}
\end{theorem}

This equivalence do not hold in higher dimension, Gamber \cite{Emilie}
proved that the implication $3\Rightarrow 1$ and $1\Rightarrow 2$ still
holds in all dimensions and the existence of a fully blocking pattern implies the almost equicontinuity of the cellular automaton.

\begin{theorem}[\protect\cite{Emilie}]
Let $F:A^{\mathbb{Z}^{d}}\rightarrow A^{\mathbb{Z}^{d}}$ with radius $r$. If there exists a fully blocking pattern of size $k^{d}$ for $F$ where $k\geq r$ then $F$ is almost equicontinuous.
\end{theorem}

\subsection{Factors}

Let $(A^{\mathbb{Z}^{d}},F)$ and $(B^{\mathbb{Z}^{d}},G)$ be two cellular automata and $\pi $ be a continuous map $\pi :A^{\mathbb{Z}^{d}}\rightarrow B^{\mathbb{Z}^{d}}$ such that $\pi \circ F=G\circ \pi .$
\begin{equation*}
\begin{tabular}{lll}
$A^{\mathbb{Z}^{d}}$ & $\overset{F}{\rightarrow }$ & $A^{\mathbb{Z}^{d}}$ \\ 
$\pi \downarrow $ &  & $\downarrow \pi $ \\ 
$B^{\mathbb{Z}^{d}}$ & $\underset{G}{\rightarrow }$ & $B^{\mathbb{Z}^{d}}$
\end{tabular}%
\end{equation*}

\begin{enumerate}
\item If $\pi $ is bijective then we say that $\pi $ is a conjugacy and the
cellular automata $(A^{\mathbb{Z}^{d}},F)$,$(B^{\mathbb{Z}^{d}},G)$\ are conjugate.

\item If\ \ $\pi $ is surjective then we say that $\pi $ is factor map and
the cellular automaton $(B^{\mathbb{Z}^{d}},G)$ is a factor of $(A^{\mathbb{Z}^{d}},F)$.
\end{enumerate}

\subsection{Measurable Dynamics}

We endow the symbolic space $A^{\mathbb{Z}^{d}}$ with the sigma-algebra \ss\ generated by all cylinder sets and $\mu $
the uniform Bernoulli measure which gives to any letter from the alphabet $A$ the same
probability. \newline
As the uniform measure is only invariant when the cellular automaton is
surjective then the following $(A^{\mathbb{Z}^{d}},$\ss $,F,\mu )$ denote a surjective cellular automaton equipped with the uniform Bernoulli measure.

An element of the unit circle is an eigenvalue of a cellular automaton $F$
associated to the eigenfunction $g$ if we have $g\circ F=\lambda g$. Whether
we place ourselves from the topological or ergodic perspective the function $
g$ is required either to be continuous or measurable. An eigenvalue $\lambda =e^{2\pi
i\alpha }$ is said irrational if $\alpha $ is an irrational number \cite%
{Walters}.

A cellular automaton $(A^{\mathbb{Z}^{d}},$\ss $,F,\mu )$ is ergodic if the measure of any invariant subset of $\mathbb{Z}^{d}$ is either $0$ or $1$. Otherwise, $F$ is ergodic if  any eigenfunction is of constant module \cite{Walters}.
\subsubsection{Gilman's Classification }

Gilman proposed a classification for cellular automata on the basis of a measurable version of the existence  of equicontinuity points.  The cellular automata considred  are equipped with a Bernoulli measure and are not necessarily surjective. 
\begin{definition}
    Let $ (A^{\mathbb{Z}}, F, \mu )$  be a cellular automaton where $\mu$ is the Bernoulli measure. \\For a point $x \in A^{\mathbb{Z}}$, we define the following set:
    \begin{center}
        $B_{[a,b]}(x)=\{y \in A^{\mathbb{Z}}: \forall n\in \mathbb{N}, F^{n}_{[a,b]}(y)=F^{n}_{[a,b]}(x)\}$ 
    \end{center}
\end{definition}
\begin{definition}
 Let $ (A^{\mathbb{Z}}, F, \mu )$  be a cellular automaton equipped with a Bernoulli measure $\mu$. 
 \begin{itemize}
     \item A point $x \in A^{\mathbb{Z}}$ is a $\mu$- equicontinuous point if for any $m>0$ we have:
 \begin{center}
     $\lim_{n\rightarrow +\infty} \frac{\mu([x_{[-n,n]}] \bigcap B_{[-m,m]}(x))}{[x_{[-n,n]}] }=1$
 \end{center}
 \item $F$ is almost expansive if there exist an integer $m>0$ such that for any point $x \in A^{\mathbb{Z}}$ , we have $\mu (B_{[-m,m]}(x))=0.$
 \end{itemize}
     
\end{definition}
\begin{proposition}
    Cellular automata $ (A^{\mathbb{Z}}, F, \mu )$ where $\mu$ is the Bernoulli measure are divided into the three following classes: 
    \begin{enumerate}
        \item $F \in \mathcal{A}$ if $F$ admits a topological  equicontinuity point.
        \item $F \in \mathcal{B}$ if $F$ admits a $\mu $ -equicontinuity points without any topological ones.
        \item $F \in \mathcal{C}$ if $F$ is almost expansive.
    \end{enumerate}

\end{proposition}

\section{Results}

\subsection{Topological results}

In this section, we focus on the set of periodic factors of a D-dimensional cellular automata with equicontinuity points that are not equicontinuous. After that we show that a D-dimensional cellular automaton cannot have topological irrational eigenvalues.

\
\begin{definition}
    A fully blocking word $w$ is of  order $p$ if  the period of the sequence $[F^{i}([w]_{\overrightarrow{0}})]_{\overrightarrow{0}} , $ i $ \in \mathbb{N}$ \\is  $p$. i.e. 
    \begin{center}
        $\forall x \in $ $A^{\mathbb{Z}^{d}}$ , $\exists m,p \in \mathbb{N} : F^{m+np}([w])_{_{\overrightarrow{0}}}=F^{m}([w])_{_{\overrightarrow{0}}} ,\forall n\in \mathbb{N}$. 
    \end{center}
\end{definition}

\begin{proposition}
Let $(A^{\mathbb{Z}^{d}},F)$ be a cellular automaton with radius $r^{d}$. If $F$ admit a fully blocking word $w$ of period $p\neq 1$ such that for any integer $n$ , $F^{n}(w)$ is fully blocking too, then $F$ has at least a non trivial periodic factor.
\end{proposition}

\begin{proof}
Let $n$ be an integer and $x$ be an equicontinuity point for $F$ admitting occurrences of the blocking pattern $w$ at the coordinates given by $\overrightarrow{i}
\in \mathbb{Z}^{d}$ such that:
\begin{equation*}
\left\Vert \overrightarrow{i}\right\Vert =\left\Vert \overrightarrow{%
(i_{1},...,i_{d})}\right\Vert =\max_{1\leq j\leq d}|i_{j}|=nr\text{ and }%
\forall j\in [1,d],\text{ }i_{j}\text{ is a divisor of } n.
\end{equation*}

The sequence of
patterns $(F^{i}(x)_{[-nr,nr+r]^{d}})_{i\in \mathbb{N}
}$ is eventually periodic. This is due to the density of the shift periodic configurations that have an eventually periodic trace and to the equicontinuity of $x$. Then, there exist integers $m,p$ such
that for all $n\in 
\mathbb{N}
:$
\begin{equation*}
F^{m}(x)_{[-nr,nr+r]^{d}}=F^{m+np}(x)_{[-nr,nr+r]^{d}}.
\end{equation*}
Consider the sets $W_{k}$ of cylinder sets defined by $W_{k}= [F^{k}(x)_{[-nr,nr+r]^{d}}]$ such that $m\leq k\leq m+p-1$ and denote by $%
W=\bigcup\limits_{k=m}^{m+p-1}W_{k}.$ By assumption, we have $\forall i\in \mathbb{N}, F^{i}(w)$ is a fully blocking pattern what means that $F^{i}(x)$ is an equicontinuity point for any integer $i$ and for $m\leq k\leq m+p-1$, we have: $F(W_{k}) \subseteq W_{k+1}$.

Let $(\mathbb{Z}/p\mathbb{Z},P=(x+1)mod p)$ be a periodic dynamical system and $\pi $ be a
function defined from $W$ to $%
\mathbb{Z}
/p%
\mathbb{Z}
$ as follow:%
\begin{equation*}
\pi (x)=k-m\text{ for }x\in W_{k}\text{ and }m\leq k\leq m+p-1
\end{equation*}
Then we have:
\begin{eqnarray*}
\pi (F(x)) &=&\left\{ 
\begin{array}{c}
k+1-m\text{ if }x\in W_{k}\text{ st. }m\leq k\leq m+p-2 \\ 
1\text{ if }x\in W_{m+p-1}\text{ \ \ \ \ \ \ \ \ \ \ \ \ \ \ \ \ \ \ \ \ \ \
\ \ \ \ \ \ \ \ \ \ \ \ \ \ \ \ \ \ \ \ }%
\end{array}%
\right. \\
P(\pi (x)) &=&(\pi (x)+1)modp=\left\{ 
\begin{array}{c}
k+1-m\text{ if }x\in W_{k}\text{ st. }m\leq k\leq m+p-2 \\ 
1\text{ if }x\in W_{m+p-1}\text{ \ \ \ \ \ \ \ \ \ \ \ \ \ \ \ \ \ \ \ \ \ \
\ \ \ \ \ \ \ \ \ \ \ \ \ \ \ \ \ \ \ \ }%
\end{array}
\right.
\end{eqnarray*}
By consequence, the restriction of $F$ to $W$ admit $(%
\mathbb{Z}
/p
\mathbb{Z}
,P=(x+1)modp)$ as a periodic factor of period $p.$
\end{proof}

\begin{proposition}
Let $(A^{
\mathbb{Z}^{d}},F)$ be a cellular automaton, then $F$ cannot have a topological
irrational eigenvalues.
\end{proposition}
\clearpage
\begin{proof}
Let $\alpha $ be a real number and $g$ be a continuous non zero function such that :
\begin{center}
    $g\circ F(x)=e^{2\pi i\alpha }g(x), \forall x \in A^{\mathbb{Z}^{d}}.$ 
\end{center}

The set of periodic points for the shift is dense in $A^{%
\mathbb{Z}
^{d}}$ then $g$ cannot be null on this set. So, there exist at least one shift periodic point $x$  such that $g(x)\neq 0$. 
Or, as the shift periodic points are eventually periodic under the iterations of $F$ then there exist integers $m,p\in \mathbb{N}$ such that: $F^{m}(x)=F^{m+ip}(x)$ for all $i \in \mathbb{N}$.
This leads to the following: 
\begin{equation*}
g(F^{m+p}(x))=e^{2\pi ip\alpha }.g(F^{m}(x))
\end{equation*}
What means that $e^{2\pi ip\alpha }=1$,by consequence $\alpha $ can only be rational.
\end{proof}

\subsection{Ergodic results}

In this part, we suppose that the cellular automaton is surjective and the space of the configurations is
endowed with the uniform Bernoulli measure $\mu .$

\begin{definition}
Let $(A^{\mathbb{Z}^{d}},F,\mu )$ be a surjective cellular automaton.\newline
Define the measurable sets $Q_{n}$ by:
\begin{eqnarray*}
Q_{n} &=&\{x\in A^{\mathbb{Z}^{d}}:\exists m,p\in \mathbb{N}:\forall i\in \mathbb{N}:F^{m+ip}(x)_{[-n,n]^{d}}=F^{m}(x)_{[-n,n]^{d}}\} \\
&=&\{x\in A^{\mathbb{Z}^{d}}:(F^{m}(x)_{[-n,n]^{d}})_{i\in \mathbb{N}}\text{ is eventually periodic}\}.
\end{eqnarray*}
\begin{eqnarray*}
Q_{n}(k) &=&\{x\in A^{\mathbb{Z}^{d}}:\exists m,p\in \mathbb{N}, 1\leq p \leq k,\forall i\in \mathbb{N}:F^{m+ip}(x)_{[-n,n]^{d}}=F^{m}(x)_{[-n,n]^{d}}\}.
\end{eqnarray*} 
The set  $Q_{n}(k)$ represents the set of elements from $Q_{n}$ with periods bounded by $k$. $Q_{n}=\bigcup_{k\in \mathbb{N}} Q_{n}(k)$.

\end{definition}

\begin{lemma}
Let $(A^{\mathbb{Z}^{d}},F,\mu )$ be a surjective cellular automaton.\newline
If $\mu (Q_{n})=1,\forall n\in 
\mathbb{N}
$ then $F$ cannot have measurable irrational eigenvalues.
\label{lemma}
\end{lemma}

\begin{proof}
Suppose that there exist $\alpha \in\mathbb{R}
$, a subset $H\subset A^{\mathbb{Z}^{d}}$ of measure $1$ and a measurable nonzero function $g$ such that: $
\forall x\in H,g(F(x))=e^{2\pi i\alpha }g(x).$\newline
For $d>0$ denote by $E_{d}$ the measurable set of elements $x$ of $A^{\mathbb{Z}^{d}}$ satisfying 
$|g(x)|>d.$ We define by $H_{n}(\eta )$ for any $\eta >0$ the measurable following set:
\begin{equation*}
H_{n}(\eta )=\{x\in H:\forall y\in
H,x_{_{[-n,n]^{d}}}=y_{_{[-n,n]^{d}}}
\Rightarrow |g(x)-g(y)|<\eta \}.
\end{equation*}

By Lusin's theorem, there exist a closed set $F \subseteq H$ such that:
\begin{center}
    
$\left\{ 
\begin{tabular}{c}
$\mu(H \setminus F) < \epsilon$, for any $\epsilon >0$
\\ Any sequence of elements $y_{n} \in F$ that converge to $x\in F$,  $g(y_{n})$ converges to $g(x)$  
\end{tabular}
\right.$
\end{center}

In particular, for any point $x\in H$, consider the  sequence of elements $y_{n}$ from $H$ that share the same coordinates with $x$ at the positions $[-n,n]^{d}$ , then
$|g(y_{n})-g(x)|< \eta $ , $\forall \eta >0$. What means that $\mu(H_{n}(\eta)) $ =$\mu (F)$ $> 1-\epsilon$ when $n$ tends to infinity.

We choose $d_{0}$ such that $\mu (E_{d_{0}})=a>0$ and take $\eta _{0}<d_{0}.$ \newline
For $0<\varepsilon <\frac{a}{2}$ there exist $m>0$ such that for any $n>m$
we have:
\begin{equation*}
 \mu
(H_{n}(\eta _{0}))>1-\frac{\varepsilon }{3}\text{ and }\mu (Q_{n})>1-\frac{%
\varepsilon }{3}
\end{equation*}
Set a value $n>m$ and let $k$ be an integer and let $Q_{n}(k)$ be the
measurable part of $Q_{n}$ where the period is bounded by $k.$ When $k$ tends to infinity $Q_{n}(k)$ tends to $Q_{n}$. In particular, for $\frac{%
\varepsilon }{3}>0$ there exists then $k_{0}$ such that for all $k\geq k_{0}$: $|\mu (Q_{n}(k))- \mu (Q_{n})|< \frac{\epsilon}{3} $. i.e,
\begin{equation*}
\mu (Q_{n}(k))\geq \mu (Q_{n})-\frac{\varepsilon }{3}\geq 1-\frac{%
2\varepsilon }{3}.
\end{equation*}%
For fixed value $k\geq k_{0}$ the sequence $F^{k}(Q_{n}(k))_{[-n,n]^{d}}$ is periodic and the period is bounded. So all the
points from $F^{k}(Q_{n}(k))_{[-n,n]^{d}}$ have a common period $p$.
\begin{equation*}
\forall x\in Q_{n}(k),\forall N\in \mathbb{N}:F^{k}(x)_{[-n,n]^{d}}=F^{k+Np}(x)_{[-n,n]^{d}}
\end{equation*}
Denote by $S=E_{d_{0}}\cap F^{-k}((H_{n}(\eta _{0})\cap Q_{n}(k)).$\newline
Now, we will show that $\mu (S)>0.$ We have:%
\begin{eqnarray*}
\mu (S) &\geq &\mu (E_{d_{0}})+\mu (F^{-k}(H_{n}(\eta _{0})\cap Q_{n}(k)))-1
\\
&\geq &\mu (E_{d_{0}})+\mu (H_{n}(\eta _{0}))+\mu (Q_{n}(k))-2 \\
&\geq &a+1-\frac{\varepsilon }{3}+1-\frac{2\varepsilon }{3}-2\geq
a-\varepsilon >\frac{a}{2}.
\end{eqnarray*}
Hence the set $S=E_{d_{0}}\cap F^{-k}(H_{n}(\eta _{0})\cap Q_{n}(k))$ is of
positive measure.\newline
By definition of the set $H_{n}(\eta _{0})$ we have:
\begin{equation}
\forall x\in S,\forall N\in \mathbb{N}
:|g(F^{k}(x))-g(F^{k+Np}(x)|<\eta _{0}  \label{me}
\end{equation}
Moreover,
\begin{equation*}
\forall x\in S,\forall N\in 
\mathbb{N}
:|g(F^{k}(x))-g(F^{k+Np}(x)|=|e^{2\pi ik\alpha }-e^{2\pi i(k+Np)\alpha
}||g(x)|=|1-e^{2\pi iNp\alpha }||g(x)|
\end{equation*}
If $\alpha $ is irrational then the sequence $\{e^{2\pi iNp\alpha },N\in 
\mathbb{N}
\}$ is dense in the unite circle. There exist then $N_{0}$ such that $%
|1-e^{2\pi iN_{0}p\alpha }|>1.$\newline
From the definition of the set $E_{d_{0}}$ we have $|g(x)|>d_{0}$ over the set $S$ and we supposed $\eta _{0}<d_{0}$ then we obtain
\begin{equation}
|g(F^{k}(x))-g(F^{k+N_{0}p}(x)|=|1-e^{2\pi iN_{0}p\alpha }||g(x)|>d_{0}>\eta
_{0}  \label{metoo}
\end{equation}
The inequalities (\ref{me}) and (\ref{metoo}) are contradictory, then $\alpha $ cannot be irrational.
\end{proof}

\begin{Remark}
    The the assumption  of the Lemma  \ref{lemma} is not equivalent to the notion of $\mu-$ equicontinuity points. Indeed, the trace of any $\mu-$ equicontinuity point is eventually periodic but opposite is not valid. Furthermore, existence of   $\mu-$ equicontinuity points is not sufficient to have $\mu(Q_{n})=1, \forall n \in \mathbb{N}$.
\end{Remark}

\begin{proposition}
Let $(A^{%
\mathbb{Z}
^{d}},F,\mu )$ be a surjective cellular automaton.\\ If $F$ is equicontinuous then its eigenvalues are all rational.
\end{proposition}

\begin{proof}
As the cellular automaton $F$ is surjective and equicontinuous then there
exist an integer $p$ such that $F^{p}=I_{d}$ \cite{Kurka}. This means that $%
\mu (Q_{n})=1,\forall n\in 
\mathbb{N}
$, then $F$ cannot have measurable irrational eigenvalues.
\end{proof}

\begin{proposition}
    Let $(A^{\mathbb{Z}^{D}},F,\mu)$ be a cellular automaton. Suppose that the measure of the set of equicontinuity points $\xi $ is strictly positive and $F$ admits a measurable eigenvalue $\alpha$ associated to an eigenfunction $g$. Then,  if  $g(x) \neq 0, \forall x \in \xi $ , $\alpha$  cannot be irrational.
    \label{propn}
\end{proposition}

\begin{proof}
Let $\alpha \in \mathbb{R}$ be an eigenvalue associated to an eigenfunction $g:A^{\mathbb{Z}^{d}} \rightarrow \mathbb{C} $  such that for any equicontinuity point $x \in A^{\mathbb{Z}^{d}} , g(x) \neq 0$ .

As the trace of any equicontinuity point $x$ is eventually periodic, this is due to the definition of equicontinuity points and the density of the shift periodic points in $A^{\mathbb{Z}^{d}}$. \newline then for any integer $n \in \mathbb{N}$ :  $\xi \subset Q_{n}$ . This leads to $\mu (Q_{n}) \geq \mu(\xi)=\beta>0, \forall n\in \mathbb{N} $. 

Consider the sets $H_{n}(\eta), E_{d}$ defined in the lemma \ref{lemma} and choose  $n \in \mathbb{N}$, $\eta_{0} <d_{0}$ such that :\\$\mu (E_{d_{0}})=1-a>0
$, $\beta -a>0$ and $\mu(H_{n}(\eta_{0}))>1-\frac{\epsilon}{2}$. Denote by $S=E_{d_{0}}\cap F^{-k}(H_{n}(\eta_{0} )\cap Q_{n}(k))$ where $\mu(Q_{n}(k))>\beta-\frac{\epsilon}{2}$.
 We have:
\begin{eqnarray*}
\mu (S) &\geq &\mu (E_{d_{0}})+\mu (F^{-k}(H_{n}(\eta_{0} )\cap  Q_{n}(k)))-1 \\
&\geq &\mu (E_{d_{0}})+\mu (H_{n}(\eta_{0} ))+\mu (Q_{n}(k))-2 \\
&\geq &1-a+1-\frac{\epsilon}{2}+\beta-\frac{\epsilon}{2}-2>\frac{
\beta -a}{2}>0.
\end{eqnarray*}
From Lemma \ref{lemma}, $\alpha$ cannot be irrational.

\end{proof}

\begin{corollary}
Let $(A^{\mathbb{Z}^{D}},F,\mu )$ be a cellular automaton. If the set of equicontinuity points contains a cylinder $[u]$ then either its eigenvalue is rational or its associated eigenfunction is null on $[u]$. 
\end{corollary}

\begin{proposition}
\ \ \ \ \ \newline
Let $(A^{\mathbb{Z}^{D}},F,\mu )$ be a surjective cellular automaton with radius $r.$\newline
If $F$ admit a fully blocking pattern $w$  then $F$ cannot have measurable irrational eigenvalues.
\end{proposition}

\begin{proof}
Let $n$ be an integer.
Consider the set $\tau (w,n)$ of points from $A^{\mathbb{Z}^{d}}$ where we will fill out with the pattern $w$ the coordinates given by $\overrightarrow{i}
\in \mathbb{Z}^{d}$ such that
\begin{equation*}
\left\Vert \overrightarrow{i}\right\Vert =\left\Vert \overrightarrow{
(i_{1},...,i_{d})}\right\Vert =\max_{1\leq j\leq D}|i_{j}|=m\text{ where } n\leq m \leq n+r.
\end{equation*}
\newline
For any pattern $u$ such that $|u|=n \times n$, the cylinder defined over the pattern $u$ filled out with $w$ belong to $\tau (w,n)$. Thus $\mu (\tau (w,n))>0$ .
We show that any point from $\tau (w,n)$ has an eventually periodic trace at the positions $[-n-r,n+r]^{d}$.
Indeed, consider a point $y\in A^{\mathbb{Z}^{d}}$ filled outside of the coordinates $[-n-r,n+r]^{d}$ with
the fully blocking patterns $w$ then $y$ by construction is an equicontinuity point that belong to $\tau (w,n)$ and for $\epsilon=2^{-nr}$, there exist $\delta > 0$ such that: $\forall x \in \tau (w,n)$: $d(F^{i}(x),F^{i}(y))\leq\epsilon.$

Recall that the values for all iterates $(F^{i})_{i\in \mathbb{N}}$ of $y$ are determined just by the values of $y$ in the
positions $[-n,n]^{d}$ then the sequence $([-n,n]^{d})_{i\in \mathbb{N}} $is eventually periodic. Moreover, for a point $x$ from $\tau (w,n)$ there exist a preperiod $m$ and a period $p$ such that:
\begin{equation*}
\forall k\in \mathbb{N}:F^{m}(x)|_{[-n,n]^{d}}=F^{m+kp}(x)|_{[-n,n]^{d}}
\end{equation*}
which means that the sequence $(F^{k}(x)_{[-n,n]^{d}})_{k\in \mathbb{N}}$ is ultimately periodic for every $x\in \tau (w,n)$.\newline
Therefore, $\tau (w,n)\subset Q_{n}$ for all $n$ and $\mu (Q_{n})>0.$\newline
From the Proposition \ref{propn}, $F$ cannot have irrational measurable eigenvalues if its associated eigenfunction is not null in $\bigcap_{n\in \mathbb{N}} Q_{n}$.
\end{proof}

\clearpage
\section{Conclusion}

The main result of this paper is about the impossibility for a cellular automaton with fully blocking patterns points to have an irrational eigenvalue in any dimension.

This condition is stronger than the condition in dimension 1 where any cellular automaton with equicontinuous or almost equicontinuous points cannot have any irrational eigenvalues.

In dimension one it is still unknown whether closing cellular automata admits or not irrational eigenvalues.
Moothathu \cite{Moothatu} showed that in dimension 1 the transitivity is equivalent to weakly mixing but there is no result about the relation between mixing and weakly mixing properties.

Dynamical properties of cellular automata are less well known in higher dimensions. There are few results about mixing properties of cellular automata in such dimensions and fewer examples. It is still unknown if the result of Moothathu still holds in higher dimensions.

We know that transitive dynamical systems on the interval are mixing, is it possible to think that dimension matter here ? In other words are one dimensional transitive cellular automata mixing ? Is it possible that this is not true in higher dimension ?.

\bibliography{articles,book,texjourn,texnique,tubguide,tugboat,type,typeset,umthsmpl,unicode,uwthesis,ws-book-sample,ws-pro-sample,xampl}

\begin{thebibliography}{9}

\bibitem{Blanchard} F. Blanchard and P. Tisseur, Some properties of cellular
automata with equicontinuous points, Ann.Inst. Henri Poincar\'{e} 36 (2000),
569-582.

\bibitem{Emilie} E. G Burkhead, \textit{Topological classification of
D-dimensional cellular automata}, Dynamical Systems Vol.24,No. 1,March
2009,45-61.

\bibitem{Chemlal} R. Chemlal, Some Spectral Properties of one dimensional Cellular Auomata, Journal of cellular automata, Volume 3-4,2018,159-172.

\bibitem{me} A. Dennunzio, P. Di Lena, L. Margara, \textit{Strictely
Temporally Periodic Points in Cellular Automata}, Vol. 90, no. Proc.
AUTOMATA\&amp; JAC 2012, 225-235.

\bibitem{Gilman} R.H. Gilman, Classes of linear automata, Ergodic
Theo.Dynam.sys., 7 (1987), 105-118.

\bibitem{Hedlund} G. Hedlund, \textit{Endomorphisms and automorphisms of the
shift dynamical systems}, Mathematical systems theory 4 (1969), n$%
{{}^\circ}%
3,$ 320--375.

\bibitem{Moothatu} T.K. Subrahmonian Moothathu : Homogeneity of surjective cellular automata. Discrete and Continuous Dynamical Systems, 13(1):195--202, June 2005.

\bibitem{Kurka} P. Kurka, \textit{Topological and Symbolic Dynamics}, Cours
sp\'{e}cialis\'{e}s, Soci\'{e}t\'{e} math\'{e}matiques De France, n$%
{{}^\circ}%
11$, Institut Henri-Poincar\'{e}, Paris, 2003.

\bibitem{Pivato} M. Pivato, The Ergodic Theory of Cellular Automata,
Departement of Mathematics, Trent University, March 6, 2007.

\bibitem{Walters} P. Walters, \textit{An introduction to ergodic theory},
Graduate Texts in Mathematics, 79, Springer, Berlin, 1982.
\end{thebibliography}

\end{document}